\documentclass{amsart}
\usepackage{graphicx,amssymb,amscd,graphics,verbatim,mathrsfs,latexsym,amsmath,amsthm,eucal}
\usepackage{latexsym}
\usepackage{color}
\usepackage{stmaryrd} 
\usepackage{amsfonts}
\usepackage{amssymb, mathrsfs, amsfonts, amsmath}
\usepackage{amsbsy}
\usepackage{amsfonts}
\newtheorem{corollary}{Corollary}

\newtheorem{definition}{Definition}
\newtheorem{lemma}{Lemma}
\newtheorem{proposition}{Proposition}

\newtheorem{theorem}{Theorem}

\numberwithin{equation}{section}
\newtheorem{question}{Question}

\theoremstyle{definition}

\numberwithin{equation}{section}

\title[Critical metrics]{Rigidity for critical metrics of the volume functional}

\author{A. Barros $^{1}$}
\author{A. Da  Silva$^{2}$}
\address{$^{1}$ Departamento de Matem\'{a}tica-Universidade Federal do Cear\'{a}\\
	60455-760-Fortaleza-CE-BR} \email{abbarros@mat.ufc.br}
\address{$^{2}$Permanent: Faculdade de Matem\'atica - Universidade Federal do Par\'a, 66075-110- Bel\'em-PA-BR.
	Current: Departamento de Matem\'{a}tica-Universidade Federal do Cear\'{a}\\
	60455-760-Fortaleza-CE-BR} \email{adamsilva@ufpa.br}
\thanks{$^{1}$ Partially supported by CNPq-Brazil}
\thanks{$^{2}$ Partially supported by CAPES-Brazil}


\begin{document}

\newcommand{\spacing}[1]{\renewcommand{\baselinestretch}{#1}\large\normalsize}
\spacing{1.2}

\begin{abstract}
Geodesic balls in a simply connected space forms $\mathbb{S}^n$, $\mathbb{R}^{n}$ or $\mathbb{H}^{n}$ are distinguished manifolds for comparison in bounded Riemannian geometry. In this paper we show that they have the maximum possible boundary volume among Miao-Tam critical metrics with connected boundary provided that the boundary of the manifold is an Einstein hypersurface. In the same spirit we also extend a rigidity theorem due to Boucher et al. \cite{Bou} and  Shen \cite{Shen} to $n$-dimensional static metrics with positive constant scalar curvature, which provides another proof of a partial answer to the Cosmic no-hair conjecture previously obtained by Chru\'sciel \cite{Chrus}.
\end{abstract}

\maketitle

\section{\textbf{Introduction}}
\label{intro}

Let $(M^{n}, g)$ be a connected, compact Riemannian manifold of dimension $n\geq 3$ with connected smooth boundary $\Sigma =\partial M$ and a fixed boundary metric $\gamma$. The purpose of this paper is to present some rigidity results for critical points of the volume functional $V: \mathcal{M}_{\gamma}^{R}\rightarrow \mathbb{R}$, where $\mathcal{M}_{\gamma}^{R}$ is the space of metrics on $M$ which has constant scalar curvature $R$ as well as induced metric on $\Sigma$ given by $\gamma$. In recent years some important results were obtained about critical points of $V$, for instance, Miao and Tam \cite{miaotam} showed that if $g\in \mathcal{M}_{\gamma}^{R}$ has the property that the first Dirichlet eigenvalue of $(n-1)\Delta_{g} + R$ is positive, then $g$ is a critical point of $V$ restricted to the space $\mathcal{M}_{\gamma}^{R}$ if and only if there exists a smooth function $f$ on $M$ satisfying the following system of PDE's
\begin{equation}\label{eqMiaoTam1}
-(\Delta_{g} f)g+\nabla^{2}_{g} f-fRic_{g}=g
\end{equation}with $f|_{\Sigma}=0$, where $Ric_{g},$ $\Delta_{g}$ and $\nabla^{2}_{g} $ stand, respectively, for the Ricci tensor, the Laplacian operator and the Hessian form associated to $g$ on $M^n$. In the same paper they also proved that among compact domains $\Omega$  in simply connected space forms, geodesic balls are the only critical metrics of $V$ ( if $\Omega\subset \mathbb{S}^{n} ,$ one assumes $V(\Omega)< \frac{1}{2}V(\mathbb{S}^{n}))$. It should be noted that height functions, up to constants, solve Eq. (\ref{eqMiaoTam1}) on geodesic balls in simply connected space forms, see e.g. \cite{BalRi2,BDRR} or \cite{miaotam}.

From now on we will follow the terminology used in \cite{BDR}. We recall the definition of Miao-Tam critical metrics.

\begin{definition}\label{DefMiaoTam}
\label{def1} A Miao-Tam critical metric is a 3-tuple $(M^n,\,g,\,f),$ where $(M^{n},\,g),$ is a compact Riemannian manifold of dimension at least three with a smooth boundary $\Sigma$ and $f: M^{n}\to \Bbb{R}$ is a smooth function such that $f^{-1}(0)=\Sigma$ satisfying the overdetermined-elliptic system
\begin{equation}
\label{eqMiaoTam} \mathfrak{L}_{g}^{*}(f)=g.
\end{equation} Here, $\mathfrak{L}_{g}^{*}$ is the formal $L^{2}$-adjoint of the linearization of the scalar curvature operator $\mathfrak{L}_{g}$ given by $\mathfrak{L}_{g}^{*}(f)=-(\Delta_{g} f)g+\nabla^{2}_{g}f-fRic_{g}$. Such a function $f$ is called potential function.
\end{definition}

We point out that if $f:M\to \Bbb{R}$  satisfies (\ref{eqMiaoTam}) then the scalar curvature of $(M^n,g)$ must be constant (cf. Proposition 2.1 in \cite{CM} and Theorem 7 in \cite{miaotam}). In fact, taking the divergence of Eq. (\ref{eqMiaoTam}) we get $f\nabla R_{g} =0.$ So, if $\nabla R_{g}(p) \neq 0$ for some $p\in M$ then $\nabla R_{g} \neq 0$ in a non-empty open set $U\subset M$. Thus $f\equiv 0$ in $U$ which can not occur by Eq. (\ref{eqMiaoTam}). Therefore, $R_{g}$ is constant on $M$. Whence, we may assume that $R_{g}=n(n-1)\varepsilon ,$ where $\varepsilon = -1, 0, 1.$ Moreover, in order to avoid strongly notation we denote the quantities related to the metric g without the subscript $g.$

Motivated by ideas developed by Kobayashi \cite{Kob}, Miao and Tam in \cite{miaotamEinst} sought answers to the following question:
\begin{question}\label{question}
Under which condition are the geodesic balls in a simply connected space form $\mathbb{R}^n$, $\mathbb{S}^{n}$ or $\mathbb{H}^n$ the only Miao-Tam critical metrics?
\end{question}

 In this context, they proved that a locally conformally flat, simply connected Miao-Tam critical metric $(M^n, g, f)$ with boundary isometric to a standard sphere $\mathbb{S}^{n-1}$ must be isometric to a geodesic ball in a simply connected space form $\mathbb{R}^n$, $\mathbb{S}^{n}$ or $\mathbb{H}^n$. In low dimensions $(n=3, 4)$, Barros et al. \cite{BDR} improved this theorem by showing that the assumption of locally conformally flat can be replaced by the weaker condition of Bach-flat. In \cite{miaotamEinst}, Miao and Tam also showed that the assumption of the boundary to be isometric to a standard sphere $\mathbb{S}^{n-1}$ can be removed provided that $(M^n, g)$ is an Einstein manifold. This theorem plays a crucial role in this article.

\begin{theorem}[Miao-Tam, \cite{miaotamEinst}]\label{thmEinst}
Let $(M^n, g, f)$ be a connected, compact Einstein Miao-Tam critical metric with smooth boundary $\Sigma$. Then $(M^n, g)$ is isometric to a geodesic ball in a simply connected space form $\mathbb{R}^n$, $\mathbb{S}^n$ or $\mathbb{H}^n$.
\end{theorem}

Recently, Baltazar and Ribeiro \cite{Balt} were able to improve Theorem \ref{thmEinst} by showing that the Einstein assumption on $(M^n, g)$ can be replaced by the parallel Ricci tensor condition, which is weaker than the former. Another recent results obtained on Miao-Tam critical metrics are volume estimates of the boundary of such manifolds. Before to comment such estimates, let us recall the following definition.

\begin{definition}\label{DefStatic}
	A complete and connected Riemannian manifold $(M^n, g)$ with a (possibly nonempty) boundary $\Sigma$ is said to be static if there exists a non-negative function $f$ on $M$ satisfying
	\begin{equation}\label{static}
	\mathfrak{L}_{g}^{*}(f)=0
	\end{equation}in $M \setminus\Sigma$ and $\Sigma = f^{-1}(0)$. In this case $(M^n, g, f)$ is called a static triple or simply a static metric.
\end{definition}

We remark that Eq. (\ref{static}) appears in General Relativity, where it defines static solutions of Einstein field equations. Corvino et al. \cite{CM} (Proposition 2.1) showed that a static metric also has scalar constant curvature $R$. When the scalar curvature is positive there exists a classic conjecture called \textit{cosmic no-hair conjecture}, formulated by Boucher et al. \cite{Bou} which claims that:

\textit{The only $n$-dimensional compact static triple $(M^n, g, f)$ with positive scalar curvature and connected boundary $\Sigma$ is given by a round hemisphere $\mathbb{S}^{n}_{+}$, where the function $f$ is taken as the height function.}

There are interesting results about this subject, for instance, Kobayashi \cite{Kob} and Lafontaine \cite{Lafont} proved independently that the conjecture is true when $(M^n , g)$ is supposed to be conformally flat; see also \cite{BalRi2} for further related results. Besides, an important answer to this question was obtained by Boucher-Gibbons-Horowitz \cite{Bou} and Shen \cite{Shen} according to the next result.

\begin{theorem}[Boucher-Gibbons-Horowitz \cite{Bou}, Shen \cite{Shen}]\label{thBoub}
	Let $(M^3, g, f)$ be a compact oriented static triple with connected boundary and scalar curvature $6$. Then $\Sigma$ is a two-sphere whose area satisfies the inequality
\begin{equation*}
|\Sigma|\leq 4\pi
\end{equation*}
with the equality holding if and only if  $M$ is isometric to the round hemisphere.
\end{theorem}

Inspired by the inequality of the last theorem, with the additional condition of the boundary isometric to a sphere $\mathbb{S}^{n-1}(r)$ of radius $r>0$, Chru\'sciel \cite{Chrus} concluded that $r\leq \sqrt{\frac{n}{\Lambda}}$, where $\Lambda >0$ denotes the cosmological constant of the static triple $(M^n, g, f)$. Moreover, if $r= \sqrt{\frac{n}{\Lambda}}$, then $M^n$ is given by a round hemisphere $\mathbb{S}^{n}_{+}(\sqrt{n/\Lambda})$.

Hijazi, Montiel and Raulot \cite{Hija} established a similar result that involves an inequality for the first eigenvalue of the induced Dirac operator of each boundary component. That is, they proved that the de Sitter spacetime minimizes the absolute value of the modes of the Dirac operator on each boundary component among all  static triples with positive scalar curvature.

Batista, Di\'ogenes, Ranieri and Ribeiro \cite{BDRR} obtained a natural extension of Theorem \ref{thBoub} to Miao-Tam critical metrics. More precisely, in \cite{BDRR} it was proved the following result: \textit{Let $(M^3, g, f)$ be a compact, oriented, Miao-Tam critical metric with connected boundary $\Sigma$ and nonnegative scalar curvature. Then $\Sigma$ is a 2-sphere and
 \begin{equation}\label{estMarcos}
 |\Sigma|\leq \frac{4\pi}{C(R)},
 \end{equation}
where $C(R)= \frac{R}{6} +\frac{1}{4|\nabla f|^2}$ is constant. Moreover, the equality in $(\ref{estMarcos})$ holds if and only if $(M^3, g)$ is isometric to a geodesic ball in a simply connected space form $\mathbb{R}^3$ or $\mathbb{S}^3$.}\\

Afterward, Barbosa et al. \cite{BLF} showed that this result also holds for  negative scalar curvature  provided that the mean curvature of the boundary satisfies $H>2$. Furthermore, they extended this result to $5$-dimensional Miao-Tam critical metrics assuming that the boundary $\Sigma$ to be Einstein. Corvino, Eichmair and Miao \cite{CM} and Yuan \cite{Yuan} also obtained interesting volume estimates for critical metrics of the volume functional.

The goal of this paper is to extend for dimension $n\ge 4$ the estimates obtained in \cite{BDRR} and  \cite{BLF} for Miao-Tam critical metrics as well as the remarkable result due to   Boucher, Gibbons and Horowitz \cite{Bou} and Shen \cite{Shen} for static metrics. Thus, with these estimates we can get important answers to question \ref{question} and to the cosmic no-hair conjecture.  Before to announce our main theorem we recall basic facts about the Yamabe constant.

Let $(M^n ,g)$ be a closed smooth manifold of dimension $n$ and denote by $[g]$ the conformal class of a Riemannian metric $g$ on $M$. The Yamabe constant $Y(M,[g])$ of $[g]$ is the infimum of the normalized total scalar curvature functional restricted to $[g]$:

\begin{equation*}
Y(M, [g]) = \inf_{\tilde{g}\in [g]} \frac{\int_{M} R_{\tilde{g}}dV_{\tilde{g}}} {(\int_{M}dV_{\tilde{g}})^{\frac{n-2}{n}}},
\end{equation*}
where $R_{\tilde{g}}$ is the scalar curvature of $\tilde{g}$ and $dV_{\tilde{g}}$ its volume element. Furthermore, the Yamabe constant can be expressed in terms of positive smooth functions (taking $\tilde{g}= \varphi^{\frac{4}{n-2}}g$, $\varphi$ positive and smooth on $M$):

\begin{equation*}
Y(M, [g])= \inf_{\varphi\in C_{+}^{\infty}(M)}\frac{\int_{M} (4\frac{n-1}{n-2}|\nabla \varphi|^2 + R\varphi^2) dV_{g} }{(\int_{M} |\varphi|^{\frac{2n}{n-2}})^{\frac{n-2}{n}}}.
\end{equation*}

The infimum always exists by a fundamental theorem obtained in several steps by Yamabe \cite{Yam}, Trudinger \cite{Trud}, Aubin \cite{Aub} and Schoen \cite{Scho}. It is not difficult to verify that the Yamabe constant of a standard sphere $\mathbb{S}^{n}$ is given by
\begin{equation}
\label{yamabesn}
Y(\mathbb{S}^{n}, [g_{can}]) = n(n-1)\omega_{n}^{2/n},
\end{equation}where $\omega_{n}$ denotes the volume of the standard unit sphere $\mathbb{S}^{n}$.

Now we are in position to state our main result. First of all we will prove that geodesic balls in a simply connected space forms $\mathbb{S}^n$, $\mathbb{R}^{n}$ or $\mathbb{H}^{n}$ have the maximum possible boundary volume among Miao-Tam critical metrics with connected boundary $\Sigma$ provided that the boundary $\Sigma$ is an Einstein hypersurface. More precisely, the following theorem holds.

\begin{theorem}\label{ThAreaEst}
	Let $(M^n, g, f), \,n\geq 4,$ be a compact, oriented, Miao-Tam critical metric with connected boundary $\Sigma$,  and scalar curvature $R=n(n-1)\varepsilon $, where $\varepsilon = -1, 0, 1$. Suppose that the boundary $\Sigma$ is an Einstein manifold with positive scalar curvature $R^{\Sigma}$. In addition, if $\varepsilon = -1,$ assume that the mean curvature of $\Sigma$ satisfies $H> n-1.$ Then we have
	\begin{equation}\label{boundM}
		|\Sigma|^{\frac{2}{n-1}} \leq \frac{Y(\mathbb{S}^{n-1}, [g_{can}])}{C(R)},
	\end{equation}
	where $C(R)=\frac{n-2}{n}R+\frac{n-2}{n-1}H^{2}$ is a positive constant. Moreover, the equality in $(\ref{boundM})$ holds if and only if $(M^n, g)$ is isometric to a geodesic ball in a simply connected space form $\mathbb{S}^n$, $\mathbb{R}^{n}$ or $\mathbb{H}^n$.
\end{theorem}

According to Theorem \ref{ThAreaEst} we obtain a partial answer to the question of Miao and Tam in \cite{miaotamEinst} (Question \ref{question}). More precisely, we have the following result.
\begin{corollary}\label{CorAnswerMiao}
Let $(M^n, g, f), \,n\geq 4,$ be a compact, oriented, Miao-Tam critical metric with connected boundary $\Sigma$ isometric to a sphere $\mathbb{S}^{n-1}(r)$ of radius $r=\Big(\frac{(n-1)(n-2)}{C(R)}\Big)^{1/2}$, and scalar curvature $R=n(n-1)\varepsilon $, where $\varepsilon = -1, 0, 1$. In addition, if $\varepsilon = -1,$ assume that the mean curvature of $\Sigma$ satisfies $H> n-1$. Then $(M^n, g)$ is isometric to a geodesic ball in a simply connected space form $\mathbb{S}^n$, $\mathbb{R}^{n}$ or $\mathbb{H}^n$.
\end{corollary}

It should be emphasized that under the normalization of the scalar curvature considered in the above corollary the constant $C(R)$ becomes
\begin{equation}
\label{cr}C(R)=\frac{n-2}{n-1}\big(H^{2}+(n-1)^2\varepsilon\big).
\end{equation}

When a Miao-Tam critical metric has nonnegative scalar curvature it is possible to estimate the volume of  $M^n$ according to the next corollary.

\begin{corollary}\label{CorEstVolM}Under the same conditions of Theorem \ref{ThAreaEst}, but $R\ge 0$,  we deduce
	\begin{equation}\label{EstVolM}
		\Big(\frac{nH}{n-1}\Big)^{\frac{2}{n-1}}|M|^{\frac{2}{n-1}}\leq \frac{Y(\mathbb{S}^{n-1}, [g_{can}])}{C(R)}.
	\end{equation} Moreover, the equality in $(\ref{EstVolM})$ holds if and only if $(M^n, g)$ is isometric to a geodesic ball in the Euclidean space $\mathbb{R}^{n}$.
\end{corollary}

For static metrics, the next theorem extends the result  due to Boucher-Gibbons-Horowitz \cite{Bou} as well as one due to  Shen \cite{Shen} for arbitrary dimension. More precisely, we will prove the following result.

\begin{theorem}\label{ThStatic}
	Let $(M^{n}, g, f),\,n\geq 4,$ be a compact, oriented static triple with connected boundary $\Sigma$, and positive scalar curvature $R=n(n-1)$. Suppose that the boundary $\Sigma$ is an Einstein manifold with positive scalar curvature $R^{\Sigma}$, then it holds
	\begin{equation}\label{EstStaticTh}
	|\Sigma|\leq \omega_{n-1}.
	\end{equation}
	In addition, the equality in $(\ref{EstStaticTh})$ is attained only for a round hemisphere $\mathbb{S}_{+}^n.$
\end{theorem}

Thus, Theorem \ref{ThStatic} provides a partial answer to the cosmic no-hair conjecture formulated by Boucher et al. \cite{Bou}. To be precise, we get an alternative proof to the following corollary previously obtained by Chru\'sciel \cite{Chrus}.

\begin{corollary}\label{CorConjec}
	The Cosmic no-hair conjecture is true when the boundary $\Sigma$ is isometric to a standard sphere $\mathbb{S}^{n-1}$.
\end{corollary}

Gibbons, Hartnoll and Pope in \cite{Gib} constructed counterexamples to the cosmic no-hair conjecture in the cases $4\leq n\leq 8$. However, in these counterexamples one can find boundary components which are topologically spherical but endowed with non-round metrics and Riemannian products of spheres. So, this does not contradict Corollary \ref{CorConjec}.

\section{\textbf{Preliminaries}}
\label{Preliminaries}

In this section we present some basic facts and results about Miao-Tam critical metrics and static triples that will be useful for the conclusion of our results announced before. First, let $(M^n, g ,f)$ be a connected Miao-Tam critical metric obeying definition \ref{DefMiaoTam}. Since $\Sigma = f^{-1}(0)$ we deduce that $f$ does not change sign on $M$. Without loss of generality may assume $f$ nonnegative, in particular, $f>0$ at the interior of $M$.

We recall that the fundamental equation (\ref{eqMiaoTam1}) can be rewritten in tensorial language as
\begin{equation}\label{eq:tensorial}
-(\Delta f)g_{ij}+\nabla_{i}\nabla_{j}f-fR_{ij}=g_{ij}.
\end{equation}

In particular, tracing this equation we have
\begin{equation}\label{eqtrace}
(n-1)\Delta f+Rf=-n,
\end{equation}
where $R$ stands for the scalar curvature on $M$. Moreover, using (\ref{eqtrace}) we can check that
\begin{equation}
\label{RicHess} f\mathring{Ric}=\mathring{\nabla^{2} f},
\end{equation}where $\mathring{T}$ stands for the traceless of $T.$

As previously stated, we already know that we must have constant scalar curvature which will be assumed $R= n(n-1)\varepsilon$, $\varepsilon = -1, 0, 1$. Note also that $|\nabla f| \neq 0$ on $\Sigma$. It is important to  point out that $f$ and $g$ are analytic according to  Corvino, Eichmair and Miao \cite{CM}. Thus, $f$ cannot vanish identically in a non-empty open set; whence the set of regular points of $f$ is dense in $M$. Moreover, since $f\geq 0$, the outward unit normal $\nu$ to the level set $\Sigma = f^{-1}(0)$ is given by $\nu=-\frac{\nabla f}{|\nabla f|}$. Further,  $|\nabla f|$ is constant on $\Sigma$. In fact, let $X\in T\Sigma$, so using (\ref{eq:tensorial}) and (\ref{eqtrace}) we get
$X(|\nabla f|^2)=2
\nabla^2 f(X, \nabla f)=-\frac{2}{n-1}\langle X, \nabla f \rangle =0.$

Proceeding, let $\{e_{1}, \ldots, e_{n-1}, \nu\}$ be an adapted orthonormal frame on $\Sigma$. The second fundamental form on  $\Sigma$ is given by
$ h_{ij}=\langle\nabla_{e_{i}}\nu, e_{j}\rangle= -\frac{1}{|\nabla f|} \langle\nabla_{e_{i}}\nabla f, e_{j}\rangle$, which gives
\begin{equation}\label{2ff1}
h_{ij}= \frac{1}{(n-1)|\nabla f|}g_{ij},
\end{equation}for $i,j=1,\ldots, n-1$. So, the boundary $\Sigma$ is totally umbilical with positive constant mean curvature given by
\begin{equation}\label{meancurv}
H=\frac{1}{|\nabla f|}.
\end{equation}

On the other hand, the Gauss equation for $\Sigma$ reads as follows
\begin{equation}
R^{\Sigma}_{ijkl}=R_{ijkl}-h_{il}h_{jk}+h_{ik}h_{jl}.
\end{equation}

Now, taking the trace in the coordinates $j,l$ and using (\ref{2ff1}) we get
\begin{equation}\label{EqGausRic}
R^{\Sigma}_{ik}=R_{ik} - R_{i\nu k\nu} + \frac{n-2}{(n-1)^2|\nabla f|^2}g_{ik}.
\end{equation}

Again, taking the trace in (\ref{EqGausRic}) we infer
\begin{equation}\label{eqGauss}
2Ric(\nu, \nu) + R^{\Sigma}= R + \frac{n-2}{n-1}H^2,
\end{equation}
where $R^{\Sigma}$ stands for the scalar curvature of $\Sigma$.

Next we deal with static triples. But, we will avoid details because the main results about such metrics follow mutatis mutandis as in the former case.

The fundamental equation (\ref{static}) for a static triple is given by
\begin{equation}\label{eqfundstatic}
-(\Delta f)g_{ij}+\nabla_{i}\nabla_{j}f-fR_{ij}=0,
\end{equation} whose trace gives
\begin{equation}\label{LaplacStatic}
\Delta f= -\frac{R}{n-1}f.
\end{equation}

Moreover, we also rewrite equation (\ref{eqfundstatic})  as follows
\begin{equation}\label{RicHesstatic}
f\mathring{Ric}=\mathring{\nabla^2 f}.
\end{equation}

We also have $|\nabla f|$ a positive constant on $\Sigma$ because $0$ is a regular value of $f$ as well as $M^n$ has constant scalar curvature $R$ (Corvino et al. \cite{CM}). Here we will  assume $R$ positive obeying $R=n(n-1)$. The difference with the former case is that the boundary $\Sigma$ now is totally geodesic. In fact, from (\ref{eqfundstatic}) and (\ref{LaplacStatic}), the second fundamental form of $\Sigma$ turns out
\begin{eqnarray}\label{secformStatic}
 h_{ij}&=& -\frac{1}{|\nabla f|}\big(R_{ij}-\frac{R}{n-1}g_{ij}\big)f ,
\end{eqnarray}for an adapted orthonormal frame $\{e_{1}, \ldots, e_{n-1}, \nu=-\frac{\nabla f}{|\nabla f|}\}$ associated to $\Sigma$. So, the Gauss equation becomes
\begin{equation}\label{eqGaussStatic}
2Ric(\nu, \nu) = R - R^{\Sigma},
\end{equation}where $R^{\Sigma}$, as before,  is the scalar curvature of $\Sigma$.

With these informations  we are in position to present the crucial results  which will be important to arrive at our main focus.

From equation (\ref{RicHess}) or \eqref{RicHesstatic} we can check that the following identity holds for the metrics considered here:
\begin{equation}
\label{divric}
f|\mathring{Ric}|^2 = \langle \mathring{Ric}, \mathring{\nabla^2 f}\rangle = div(\mathring{Ric}(\nabla f)).
 \end{equation}Thus, integrating \eqref{divric} on $M$, using that $|\nabla f|$ is constant on $\Sigma$, and applying the divergence theorem we deduce the following lemma, which was proved previously in \cite{BDRR}.

\begin{lemma}\label{intfE}Let $(M^n ,g, f)$ be a compact, oriented, connected Miao-tam critical metric (or static triple) with connected smooth boundary $\Sigma$. Then,
	\begin{equation*}
	\int_{M}f|\mathring{Ric}|^2 dV_{g} = -\frac{1}{|\nabla f|}\int_{\Sigma}\mathring{Ric}(\nabla f, \nabla f) ds.
	\end{equation*}
\end{lemma}

In order to prove Theorem \ref{ThAreaEst} we need the following proposition which can be found in \cite{BLF}. But, for sake of completeness we include its proof here.
\begin{proposition}\label{intRbound}
	Let $(M^n, g, f)$, $n\geq 3$, be a compact, oriented, connected Miao-Tam critical metric with connected smooth boundary $\Sigma$ and scalar curvature $R= n(n-1)\varepsilon$, where $\varepsilon = -1, 0, 1$. Then the following identity occurs
\begin{equation*}
	\int_{\Sigma} R^{\Sigma} ds = 2H \int_{M} f|\mathring{Ric}|^2 dV_{g} + C(R)|\Sigma|,
	\end{equation*}where $C(R)$ is given by \eqref{cr}.
\end{proposition}
\begin{proof}
	From Lemma \ref{intfE}  and Gauss equation (\ref{eqGauss}) we get
	\begin{eqnarray*}
	\int_{M}f|\mathring{Ric}|^2 dV_{g} &=& -\frac{1}{H}\int_{\Sigma}Ric(\nu, \nu) dV_{g} + \frac{1}{H}\int_{\Sigma}\varepsilon (n-1) ds\\
	&=&-\frac{1}{H}\int_{\Sigma}\frac{1}{2}\Big(\varepsilon n(n-1) - R^{\Sigma} +\frac{n-2}{n-1}H^2\Big)ds + \frac{1}{H}\varepsilon (n-1)|\Sigma|\\
	&=&\frac{1}{2H}\int_{\Sigma}R^{\Sigma} ds - \frac{1}{2H}\Big(\varepsilon (n-1)(n-2)+\frac{n-2}{n-1}H^2\Big)|\Sigma|.
	\end{eqnarray*}
	Therefore,
	$$\int_{\Sigma} R^{\Sigma} ds = 2H \int_{M} f|\mathring{Ric}|^2 dV_{g} + \frac{n-2}{n-1}\big(H^2+\varepsilon (n-1)^2\big)|\Sigma|,$$ as desired.
\end{proof}

Applying again  Gauss equation (\ref{eqGaussStatic}) in Lemma \ref{intfE} we deduce a similar result to static triple.
\begin{proposition}\label{intRboundStatic}
	Let $(M^n, g, f)$, $n\geq 3$, be a compact, oriented, connected static triple with connected smooth boundary $\Sigma$ and scalar curvature $R=n(n-1)$. Then we deduce
	\begin{equation*}
	\int_{\Sigma} R^{\Sigma} ds = \frac{2}{|\nabla f|}\int_{M} f|\mathring{Ric}|^2 dV_{g} + (n-1)(n-2)|\Sigma|.
	\end{equation*}
\end{proposition}

Finally, we announce one of the fundamental source due to Ilias \cite{Illias} that will be used in the proof of our  main theorem.
Before, let us introduce the following constants:
\begin{equation}\label{kn}
	K(n,2)= \sqrt{\frac{4}{n(n-2)\omega_{n}^{2/n}}},
	\end{equation}where  $\omega_{n} = |\mathbb{S}^n|,$ as well as
\begin{equation}
\label{infric}\mathcal{R}(M, g) = \inf\{Ric(V, V) \;|\; V\in TM, |V|_{g}=1\}.
\end{equation} Then we get the following result.

\begin{theorem}\cite{Illias}\label{thIlias}
	Let $(M^n , g),\,n\ge 3,$ be a compact riemannian manifold without boundary. Suppose that $\mathcal{R}(M, g)\geq \mathcal{R}(\mathbb{S}^n , \frac{1}{\delta}g_{can})=(n-1)\delta>0$, then
	\begin{equation*}
	\Big(\int_{M} |f|^{\frac{2n}{n-2}} dV_{g}\Big)^{\frac{n-2}{n}}\leq [K(n,2)]^2 \Big(\frac{\omega_{n}(\delta)}{|M|}\Big)^{\frac{2}{n}} \int_{M}|\nabla f|^2 dV_{g} + |M|^{-\frac{2}{n}}\int_{M} |f|^2 dV_{g},
	\end{equation*}
	for all $f\in H^{1,2}(M)$, where $\omega_{n}(\delta)=\delta^{-\frac{n}{2}}\omega_{n}$.
	
\end{theorem}

\section{Proofs of the Main Results}

\subsection{\textbf{Proof of Theorem \ref{ThAreaEst}}}

In order to prove Theorem \ref{ThAreaEst} the previous result due to Ilias (\cite{Illias}, Theorem 3) is fundamental. Indeed, as the boundary $\Sigma$ is an Einstein manifold with $R^{\Sigma}>0$, we take $\delta=\frac{R^{\Sigma}}{(n-1)(n-2)}>0$ to obtain
 $$\mathcal{R}(\Sigma, g_{\Sigma}) = \inf\{Ric^{\Sigma}(V, V) \;|\; V\in T\Sigma, |V|_{g}=1\}= (n-2)\delta.$$ Therefore,  applying  to $\Sigma^{n-1}$ the quoted theorem due to  Ilias we deduce
\begin{eqnarray}\label{IneqSob}
&&\\
\nonumber\Big(\int_{\Sigma} |\varphi|^{\frac{2(n-1)}{n-3}} ds\Big)^{\frac{n-3}{n-1}}&\leq& [K(n-1,2)]^2 \Big(\frac{\omega_{n-1}(\delta)}{|\Sigma|}\Big)^{\frac{2}{n-1}} \int_{\Sigma}|\nabla \varphi|^2 ds + |\Sigma|^{-\frac{2}{n-1}}\int_{\Sigma} |\varphi|^2 ds,
\end{eqnarray}for all $\varphi\in H^{1,2}(\Sigma)$, where $\omega_{n-1}(\delta)=\delta^{-\frac{n-1}{2}}\omega_{n-1}$, $\omega_{n-1} = |\mathbb{S}^{n-1}|$ and $K(n-1,2)$ obeying \eqref{kn} is the best constant for inequalities of the type:
\begin{equation}\label{sob1}
\Big(\int_{\Sigma} |\varphi|^{p} ds\Big)^{\frac{1}{p}}\leq A \Big(\int_{\Sigma}|\nabla \varphi|^{q} ds\Big)^{\frac{1}{q}} + B\Big(\int_{\Sigma} |\varphi|^{q} ds\Big)^{\frac{1}{q}},
\end{equation}where $\frac{1}{p}=\frac{1}{q}-\frac{1}{n-1}$, $1\leq q <n-1$ and $q\in\mathbb{R}$. From (\ref{IneqSob}) we have

\begin{eqnarray*}
\Big(\int_{\Sigma} |\varphi|^{\frac{2(n-1)}{n-3}} ds\Big)^{\frac{n-3}{2(n-1)}}\leq [K(n-1,2)] \Big(\frac{\omega_{n-1}(\delta)}{|\Sigma|}\Big)^{\frac{1}{n-1}}\Big( \int_{\Sigma}|\nabla \varphi|^2 ds\Big)^{\frac{1}{2}} + |\Sigma|^{-\frac{1}{n-1}}\Big(\int_{\Sigma} |\varphi|^2 ds\Big)^{\frac{1}{2}}.
\end{eqnarray*}

Next we compare this last inequality with relation \eqref{sob1} to infer that the coefficient of $K(n-1,2)$ must be greater than or equal to $1$, i.e. $\Big(\frac{\omega_{n-1}(\delta)}{|\Sigma|}\Big)^{\frac{1}{n-1}} \geq 1$. Taking into account that  $\omega_{n-1}(\delta)=\delta^{-\frac{n-1}{2}}\omega_{n-1}$ we arrive at
\begin{equation}\label{ineqwsigma}
(\omega_{n-1})^{\frac{2}{n-1}} \geq \delta |\Sigma|^{\frac{2}{n-1}}.
\end{equation}

Thus, substituting $\delta=\frac{R^{\Sigma}}{(n-1)(n-2)}$ in Eq. (\ref{ineqwsigma}) we have
$(n-1)(n-2)(\omega_{n-1})^{\frac{2}{n-1}}\geq R^{\Sigma} |\Sigma|^{\frac{2}{n-1}}.$ Whence, using \eqref{yamabesn} we deduce
\begin{eqnarray}\label{ineqYam}
	Y(\mathbb{S}^{n-1}, [g_{can}]) \geq R^{\Sigma} |\Sigma|^{\frac{2}{n-1}}.
\end{eqnarray}

Now, on integrating (\ref{ineqYam}) on $\Sigma$ and using Proposition \ref{intRbound} we conclude
\begin{eqnarray}\label{ineq}
\nonumber Y(\mathbb{S}^{n-1}, [g_{can}])|\Sigma|^{\frac{n-3}{n-1}} &\geq& \int_{\Sigma} R^{\Sigma} ds\\
 \nonumber&=& 2H \int_{M} f|\mathring{Ric}|^2 dV_{g} + \frac{n-2}{n-1}\big( H^2+\varepsilon (n-1)^2 \big)|\Sigma|\\
 &\geq& C(R)|\Sigma|.
\end{eqnarray}

Since $C(R)>0$ we conclude that
\begin{equation}\label{ineqSigY}
|\Sigma|^{\frac{2}{n-1}} \leq \frac{Y(\mathbb{S}^{n-1}, [g_{can}])}{C(R)}.
\end{equation}

Finally, from (\ref{ineq}) note that occurring  equality in (\ref{ineqSigY}) we must have  $$\int_{M} f|\mathring{Ric}|^2 dV_{g}=0.$$ That is, $(M^n,g)$ is an Einstein manifold. Then we arrive at the assumptions of Theorem \ref{thmEinst} to conclude that $M^n$ is isometric to a geodesic ball in a simply connected space form $\mathbb{R}^{n}$, $\mathbb{S}^{n}$ or $\mathbb{H}^{n}$. Conversely, if $(M^n,g)$ is isometric to a geodesic ball in a simply connected space form $\mathbb{R}^{n}$, $\mathbb{S}^{n}$ or $\mathbb{H}^{n}$, then $|\mathring{Ric}|=0$ and $\Sigma$ is a sphere $\mathbb{S}^{n-1}(r)$ of radius $r$. From Proposition \ref{intRbound} we must have
\begin{equation}
R^{\Sigma}=(n-1)(n-2)\varepsilon +\frac{n-2}{n-1}H^2.
\end{equation}
Whence
\begin{equation}\label{RicSigma}
Ric^{\Sigma}=\frac{R^\Sigma}{n-1}g_{\Sigma}=(n-2)\Big(\frac{H^2+(n-1)^2\varepsilon}{(n-1)^2} \Big)g_{\Sigma}.
\end{equation}

Therefore, from (\ref{RicSigma}) we have $\delta =\frac{H^2+ (n-1)^2\varepsilon}{(n-1)^2}$ and $r=\delta^{-1/2}$. Thus we infer
\begin{eqnarray*}
C(R)|\Sigma|^{\frac{2}{n-1}}&=& (n-2)\big(\frac{H^2+(n-1)^2\varepsilon}{(n-1)} \big)|\mathbb{S}^{n-1}(r)|^{\frac{2}{n-1}}\\
&=& (n-2)\big(\frac{H^2+(n-1)^2\varepsilon}{(n-1)} \big)r^2 (\omega_{n-1})^{\frac{2}{n-1}}\\
&=&(n-2)\big(\frac{H^2+(n-1)^2\varepsilon}{(n-1)} \big)\delta^{-1} (\omega_{n-1})^{\frac{2}{n-1}}\\
&=&Y(\mathbb{S}^{n-1}, [g_{can}]),
\end{eqnarray*}
as desired.

\subsection{\textbf{Proof of Corollary \ref{CorAnswerMiao}}}

It suffices to note that the choice of the radius  $$r=\Big(\frac{(n-1)(n-2)}{C(R)}\Big)^{1/2}= \Big(\frac{(n-1)^2}{H^2 +(n-1)^2 \varepsilon}\Big)^{1/2}$$ gives $r=\delta^{-1/2}$, which implies  equality in  the previous theorem.

\subsection{\textbf{Proof of Corollary \ref{CorEstVolM}}}

On integrating \eqref{eqtrace} on $M$, and using equation \eqref{meancurv} we arrive at the following identity
\begin{eqnarray*}
\nonumber |\Sigma| &=& \frac{nH}{n-1}|M| + \frac{RH}{n-1}\int_{M} f dV_{g}.\\
\end{eqnarray*}

Whence, since $R\ge 0$ we derive the next inequality
\begin{equation}\label{DesSigma}
|\Sigma| \geq\frac{nH}{n-1}|M|.
\end{equation}Moreover, equality is attained  if and only if $R=0$, since $H>0$ as well as $\int_{M} f dV_{g}>0.$

Now we claim that
\begin{equation}\label{EstVolM2}
\Big(\frac{nH}{n-1}\Big)^{\frac{2}{n-1}}|M|^{\frac{2}{n-1}}\le |\Sigma|^{\frac{2}{n-1}}\leq \frac{Y(\mathbb{S}^{n-1}, [g_{can}])}{C(R)}.
\end{equation}
Indeed, the first inequality follows from (\ref{DesSigma}), whereas the second one comes from Theorem \ref{ThAreaEst}. So, we establish the inequality of the corollary.
Moreover, equality occurs  in the quoted corollary if and only if it also occurs in each step of the last chain of inequalities. Now, it suffices to use identity \eqref{DesSigma}, and Theorem \ref{ThAreaEst} to conclude the proof of the corollary.

\subsection{\textbf{Proof of Theorem \ref{ThStatic}}}

The proof of Theorem \ref{ThStatic} is very similar to that one presented in Theorem \ref{ThAreaEst}. Indeed, the assumption on the boundary $\Sigma$, as well as the steps used in the proof of Theorem \ref{ThAreaEst}, allow us to conclude that
\begin{eqnarray}\label{estOmegSigmStatic}
	(n-1)(n-2)(\omega_{n-1})^{\frac{2}{n-1}}\geq R^{\Sigma} |\Sigma|^{\frac{2}{n-1}}.
\end{eqnarray}

On integrating (\ref{estOmegSigmStatic}) on $\Sigma$  we have
\begin{equation*}
(n-1)(n-2)(\omega_{n-1})^{\frac{2}{n-1}}|\Sigma|^{\frac{n-3}{n-1}}\geq \int_{\Sigma} R^{\Sigma} ds.\\
\end{equation*}
According to Proposition \ref{intRboundStatic}, we have $ \int_{\Sigma} R^{\Sigma} ds=\frac{2}{|\nabla f|}\int_{M} f|\mathring{Ric}|^2 dV_{g} + (n-1)(n-2)|\Sigma|.$ So, we arrive at
\begin{equation}
\label{ineqthm4a}
(n-1)(n-2)(\omega_{n-1})^{\frac{2}{n-1}}|\Sigma|^{\frac{n-3}{n-1}}\ge \frac{2}{|\nabla f|}\int_{M} f|\mathring{Ric}|^2 dV_{g} + (n-1)(n-2)|\Sigma|.
\end{equation}
Once more, we use $\int_{M} f|\mathring{Ric}|^2 dV_{g}\ge0$, to deduce
\begin{equation}\label{ineqthm4}|\Sigma|\leq \omega_{n-1}.
\end{equation}

Moreover, if $|\Sigma|= \omega_{n-1}$, then from \eqref{ineqthm4a} we must have
\begin{equation*}
\frac{2}{|\nabla f|}\int_{M} f|\mathring{Ric}|^2 dV_{g}=0,
\end{equation*}i.e., $(M^n, g)$ is an Einstein manifold. Next we use equation (\ref{RicHesstatic}) to arrive at the following boundary problem:
\begin{equation}
\label{pvi1}
\mathring{\nabla^2 f}\big|_{M^n}=0,  \qquad f|_{\Sigma}=0.
\end{equation}

But, using Theorem B in \cite{Reilly} we conclude that $(M^n, g)$ is isometric to a round hemisphere $(\mathbb{S}^{n}_{+}, g_{can})$ as desired.


\begin{thebibliography}{BB}
	
	
	\bibitem{Aub} Aubin, T. \textit{Equations diff\'erentielles non-lin\'eaires et probl\`eme de Yamabe concernant la courbure scalaire}. J. Math. Pures Appl. \textbf{55} (1976), 269-296.
	

\bibitem{BalRi2} Baltazar, H. and Ribeiro Jr., E: \textit{Remarks on critical metrics of the scalar curvature and vo\-lu\-me functionals on compact manifolds with boundary}. (2017), arXiv:1703.01819 [math.DG].
	
	\bibitem{Balt} Baltazar, H. and Ribeiro Jr., E. \textit{Critical metrics of the volume functional on manifolds with boundary}. Proc. Amer. Math. Soc. 145(2017) 3513-3523.
	
	\bibitem{BLF} Barbosa, E., Lima, L. and Freitas, A. \textit{The generalized Pohozaev-Schoen identity and some geometric applications}. arXiv:1607.03073v1 [math.DG] (2016).
	
	
	\bibitem{BDR} Barros, A., Di\'ogenes, R. and Ribeiro Jr., E. \textit{Bach-Flat critical metrics of the volume functional on 4-dimensional manifolds with boundary}. Journal of Geom. Analysis. \textbf{25} (2015), 2698-2715.
	
	\bibitem{BDRR} Batista, R., Di\'ogenes, R., Ranieri, M. and Ribeiro Jr., E. \textit{Critical metrics of the volume functional on compact three-manifolds with smooth boundary}. Journal of Geom. Analysis. 27 (2017) 1530-1547.
	
	
	\bibitem{Bou} Boucher, W., Gibbons, G. and Horowitz, G.  \textit{Uniqueness theorem for anti-de Sitter spacetime}. Phys. Rev. D (3) \textbf{30} (1984), 2447-2451.
	
	\bibitem{Chrus} Chru\'sciel, P. T. \textit{Remarks on rigidity of the de Sitter metric}. [homepage.univie.ac.at/piotr.chrusciel/ papers/deSitter/deSitter2] (unpublished)
	
	\bibitem{CM} Corvino, J., Eichmair, M. and Miao, P. \textit{Deformation of scalar curvature and volume}. Math. Annalen. \textbf{357} (2013), 551-584.
	
	\bibitem{Gib} Gibbons, G. W., Hartnoll, S. A. and Pope, C. N. \textit{Bohm and Einstein-Sasaki metrics, black holes and cosmological event horizons}. Phys. Rev. D \textbf{67} (2003).
	
	\bibitem{Hija} Hijazi, O., Montiel S. and Raulot, S. \textit{Uniqueness of de Sitter spacetime among static vacua with positive cosmological constant}. Ann. Glob. Anal. Geom. \textbf{47} (2015), no.2, 167-178.
	
\bibitem{Illias} Ilias, S. \textit{Constantes explicites pour les in\'egalit\'es de Sobolev sur les vari\'e t\'es riemanniennes compactes}. Annales de l'institut Fourier, tome \textbf{33} (1983), no. 2,  151-165.

\bibitem{Kob} Kobayashi, O. \textit{A differential equation arising from scalar curvature function}. J. Math. Soc. Jpn. \textbf{34} (1982), 665-675.
	
\bibitem{Lafont} Lafontaine, J. \textit{Sur la g\'eom\'etrie d'une g\'en\'eralisation de l'\'equation diff\'erentielle d'Obata}. J. Math. Pures Appl., \textbf{9} (1983), 63.	
	
	
	\bibitem{miaotam} Miao, P. and  Tam, L.-F. \textit{On the volume functional of compact manifolds with boundary with constant scalar curvature}. Calc. Var. PDE. \textbf{36} (2009), 141-171.
	
	\bibitem{miaotamEinst} Miao, P. and  Tam, L.-F. \textit{Einstein and conformally flat critical metrics of the volume functional}. Trans. Amer. Math. Soc. \textbf{363} (2011), 2907-2937.
	
	
	
	\bibitem{Reilly}  Reilly, R. \textit{Geometric applications of the solvability of Neumann problems on a Riemannian manifold}. Arch. Ration. Mech. Anal.
	\textbf{75} (1980), (1), 23-29.
	
	\bibitem{Scho} Schoen, R. \textit{Conformal deformation of a Riemannian metric to constant scalar curvature}.  J. Differential Geom. \textbf{20} (1984), 479-495.
	
	\bibitem{Shen} Shen, Y. \textit{A note on Fischer-Marsden's conjecture}.  Proc. Am. Math. Soc. \textbf{125} (1997), 901-905.
	
	\bibitem{Trud} Trudinger, N. \textit{Remarks concerning the conformal deformation of Riemannian structures on compact manifolds}. Ann. Scuola Norm. Sup. Pisa. \textbf{22} (1968), 265-274.
	
	
	\bibitem{Yam} Yamabe, H. \textit{On a deformation of Riemannian structures on compact manifolds}. Osaka Math. J. \textbf{12} (1960), 21-37.
	
	\bibitem{Yuan} Yuan, W.: \textit{Volume comparison with respect to scalar curvature}. arXiv:1609.08849v1 [math.DG] (2016).

	
\end{thebibliography}
\end{document}